\documentclass[11pt]{amsart}
\usepackage[english]{babel} 
\usepackage[latin1]{inputenc} 
\usepackage{amssymb,geometry,color,graphicx}
\usepackage{amsmath,amsthm} 
\usepackage{booktabs}
\usepackage{datetime}
\usepackage{dsfont}  
\usepackage{amstext} 
\usepackage{graphicx}   
\usepackage{fancyhdr} 
\usepackage{multirow}
\usepackage{latexsym}
\usepackage{mathrsfs}
\usepackage{listings}
\usepackage{cancel}
\usepackage{cases}
\usepackage{mathtools}
\usepackage{bm}
\usepackage{caption}
\usepackage{subcaption}
\usepackage{tikz}
\usepackage{hyperref}
\usepackage{enumitem}
\hypersetup{
    colorlinks=true, 
    linktoc=all,     
    linkcolor=blue,  
}

\usepackage{changes}
\definechangesauthor[name=Pawel, color=blue]{PP}

\numberwithin{equation}{section} 

\newcommand{\R}{\mathbb{R}}

\newcommand{\N}{\mathbb{N}}

\newcommand{\A}{\mathcal{A}}
\newcommand{\de}{\mathrm{d}}

\newcommand{\n}{{(n)}}
\newcommand{\E}{{\mathbb{E}}}
\newcommand{\OB}{{\Omega_B}}

\newcommand{\EB}{{\mathbb{E}_B}}

\newcommand{\F}{{\mathcal{F}}} 
\renewcommand{\P}{{\mathbb{P}}} 
\newcommand{\PB}{{\mathbb{P}_B}}

\newcommand{\diff}[1]{\,\mathrm{d}#1}
\newcommand{\Cb}{\mathcal{C}_b^{\alpha,\beta}}

\newcommand{\cP}{\mathcal{P}}

\newcommand{\triple}{{\vert\kern-0.25ex\vert\kern-0.25ex\vert}}

\theoremstyle{plain}
\newtheorem{thm}{Theorem}[section]
\newtheorem{cor}[thm]{Corollary}
\newtheorem{prop}[thm]{Proposition}
\newtheorem{lem}[thm]{Lemma}

\newtheorem{rmk}[thm]{Remark}

\usepackage{soul}

\newcommand{\eps}{\varepsilon}

    \title[Randomised EM for SDEs with H\"older conditions]{Randomised Euler-Maruyama method for SDEs with H\"older continuous drift coefficient}
\author[J. Bao]{Jianhai Bao}
\address{Center for Applied Mathematics, Tianjin University, 300072 Tianjin, P.R. China}
\email{jianhaibao@tju.edu.cn}
\author[Y. Wu]{Yue Wu}
\address{Department of Mathematics and Statistics, University of Strathclyde, Glasgow, G1 1XH, UK}
\email{yue.wu@strath.ac.uk,corresponding author  }

\begin{document}
\begin{abstract}
     In this paper, we examine the performance of randomised Euler-Maruyama (EM) method for additive time-inhomogeneous SDEs with an irregular drift. In particular, the drift is assumed to be $\alpha$-H\"older continuous in time and bounded $\beta$-H\"older continuous in space with $\alpha,\beta\in (0,1]$. The strong order of convergence of the randomised EM in $L^p$-norm is shown to be $1/2+(\alpha \wedge (\beta/2))-\epsilon$ for an arbitrary $\epsilon\in (0,1/2)$, higher than the one of standard EM, which is $\alpha \wedge (1/2+\beta/2-\epsilon)$. The proofs highly rely on the stochastic sewing lemma, where we also provide an alternative proof when handling time irregularity for a comparison. 
\end{abstract}
\maketitle

\section{Introduction}\label{sec:intro}
It is widely known that the uniqueness of the solution to an ordinary differential equation subject to an H\"older continuous drift is not guaranteed. However, when regularised by Brownian motion (BM), the corresponding additive SDE ensures a unique solution \cite{Zvonkin1974, Veretennikov1981, Davie2007}.
Building on existing results regarding solution uniqueness, this paper focuses on the numerical approximation for the following $\mathbb{R}^d$-valued additive time-inhomogeneous SDE, 
\begin{equation}
\label{eq:SDE}
\begin{cases}
\mathrm{d}X(t)=f(t,X(t))\mathrm{d}t+B(t), & t\in [0,T], \\
x(0)=x_0\in \mathbb{R}^d, &
\end{cases}
\end{equation}
where $B=(B(t))_{0\leq t\leq T}$ is a $d$-dimensional BM on a probability space $(\Omega_B, \mathcal{F}^B,\mathbb{P}_B)$, and the drift coefficient $f:[0,T]\times \mathbb{R}^d\to \mathbb{R}^d$ is assumed to be $\alpha$-H\"older continuous in time and bounded $\beta$-H\"older continuous in space with $\alpha,\beta\in (0,1]$, i.e., 
\begin{align}\label{eqn: fholder}
    \begin{split}
           \|f\|_{\mathcal{C}_b^{\alpha,\beta}([0,T]\times \mathbb{R}^d; \mathbb{R}^d)}&:=\sup_{t\in [0,T],x\in \mathbb{R}^d}|f(t,x)|+\sup_{x\in \mathbb{R}^d,s\neq t}\frac{|f(t,x)-f(s,x)|}{|t-s|^\alpha}\\
           &\qquad+\sup_{t\in [0,T],x\neq y}\frac{|f(t,x)-f(t,y)|}{|x-y|^\beta}<\infty.
    \end{split}
\end{align}
We will use $\mathcal{C}_b^{\alpha,\beta}$ for short. For simplicity, let $T=1$.

The widely studied scheme for the SDE \eqref{eq:SDE} is the Euler-Maruyama method (EM) on a fixed stepsize $1/n$, $n\in \N$, with the numerical solution given by
\begin{equation}\label{eqn:emcontinuous}
    \bar{X}^{(n)}_t=x_0+\int_0^t f\big(\kappa_n(s),\bar{X}^{(n)}_{\kappa_n(s)}\big)\,\de s  +B(t),
\end{equation}
where $\kappa_n(s):=\lfloor   ns \rfloor /n $ with $\lfloor   a\rfloor $ representing the largest integer that does not exceed $a$. 
The convergence in probability to $X(t)$ in \eqref{eq:SDE} of $\bar{X}^{(n)}_t$ in \eqref{eqn:emcontinuous} is established in \cite{Gyongy1996}. In the case of $\alpha \in [1/2,1)$ and $\beta\in (0,1)$, the strong order of convergence in $L^p$-norm is first given at $\beta /2$ via a PDE approach \cite{Pamen2017}. The regularity of the solution to the associated Kolmogorov equation is leveraged in the proofs of main theorems in \cite{Pamen2017} so that Gr\"onwall inequality can be used. The order of convergence merely depends on the regularity of the numerical scheme $\bar{X}_t$ and the roughness of $f$ through two types of error terms:
\begin{equation}\label{eqn:introterm1}
   \mathbb{I}_1(\bar{X}^{(n)}):= \E\left[\sup_{s\leq u\leq t}\left|\int_{s}^u \left( f(r, \bar{X}_r^{(n)}) - f(\kappa_n(r), \bar{X}_{\kappa_n(r)}^{(n)}) \right)\,\de r\right|^p\right]
\end{equation}
and
\begin{equation}\label{eqn:introterm2}
    \mathbb{I}_2(\bar{X}^{(n)},g):=\E\left[\sup_{s\leq u\leq t}\left|\int_{s}^u g(r,\bar{X}^{\n}_{r}) \cdot  \left(f(r, \bar{X}_{r}^{(n)})- f(\kappa_n(r), \bar{X}_{\kappa_n(r)}^{(n)}) \right) \,\de r\right|^p\right],
\end{equation}
where $g\in \mathcal{C}^{0,1}_b([s,t]\times \R^d; \R^d)$, equipped with the norm
$$\|f\|_{\mathcal{C}_b^{0,1}([0,T]\times \mathbb{R}^d; \mathbb{R}^d)}:=\sup_{t\in [0,T],x\in \mathbb{R}^d}|f(t,x)|
          +\sup_{t\in [0,T],x\neq y}\frac{|f(t,x)-f(t,y)|}{|x-y|}.$$

 Because of the boundedness of $f$, the local regularity of the numerical scheme $\bar{X}_t$ is easily bounded in \cite{Pamen2017} by $$\|\bar{X}_t-\bar{X}_{\kappa_n(t)}\|_{L^p(\Omega_B;\R^d)}\leq C n^{-1/2}, $$
 where $C$ is a generic constant. Thus, for $\alpha \in [1/2,1)$ and $\beta\in (0,1)$,  both $\mathbb{I}_1$ and $\mathbb{I}_2$ are estimated via a straightforward calculation as follows
\begin{align}
    \begin{split}
          &\|f(r,\bar{X}_r)-f(\kappa_n(r),\bar{X}_{\kappa_n(r)})\|_{L^p(\Omega_B;\R^d)}\\
    &\leq  \|f\|_{\Cb} (|t-\kappa_n(t)|^{\alpha}+\|\bar{X}_t-\bar{X}_{\kappa_n(t)}\|^{\beta}_{L^p(\Omega_B;\R^d)})\\
    &\leq C\|f\|_{\Cb} (n^{-\alpha}+n^{-\beta/2})\leq C n^{-\beta/2},
    \end{split}
\end{align}
where the order of convergence $\beta/2$ shows up. 
Note that when $\alpha \in (0,1/2)$, the order of convergence through the method mentioned above would be $\alpha \wedge (\beta/2)$.

When SDE is time-homogeneous with $f$ being bounded $\beta$-H\"older, a recent work \cite{Butkovsky2021} shows that the EM method is able to converge to the exact solution at a strong order $(1+\beta)/2-\epsilon$, for an arbitrary $\epsilon\in (0,1/2)$, almost optimal in the sense of lower error bounds \cite{ellinger2025}. Though the analysis strategy is very different from \cite{Pamen2017}, heavily relying on stochastic sewing lemma \cite{Le2020}, it is still the term $\mathbb{I}_1(\bar{X})$ that determines the order of convergence. For this, the authors manage to show that the upper bound of $\mathbb{I}_1(B)$ is $C n^{-(1+\beta)/2+\epsilon}$ by utilising heat kernel estimates in stochastic sewing lemma and then applying Girsanov's theorem for $\mathbb{I}_1(\bar{X})$. Later, the error analysis strategy for EM method in \cite{Butkovsky2021} are extended to an additive SDE with a drift of Sobolev regularity \cite{Dareiotis2023} and an additive SDE with distributional drift \cite{Goudenege2023}. In the meanwhile, with the observation that the error of the EM scheme can be reduced to a quadrature problem (for instance, via a Zvonkin-type transformation \cite{Zvonkin1974}), the same order of convergence as \cite{Dareiotis2023} is achieved via a non-equidistant EM method when $f$ can be decomposed into two parts: one part is twice differentiable and another one is with Sobolev regularity \cite{Neuenkirch2021}. Indeed, $\mathbb{I}_1$ and $\mathbb{I}_2$ are two types of quadrature errors.

We want to stress that, the upper bounds for $\mathbb{I}_1$ and $\mathbb{I}_2$ are limited to $Cn^{-\alpha}$ for time-inhomogeneous case, i.e., SDE \eqref{eq:SDE}.  As a consequence, the strong convergence of the EM method can not exceed $\alpha$, or more precisely, $\alpha \wedge ((\beta+1)/2-\epsilon)$. This motives us to examine the performance of randomised EM
method in this setting, which, to handle the low convergence due to time-irregular drift, artificially introduces randomness to the time variable and constructs martingale process from the random time.

In literature, for ODEs with $f\in \mathcal{C}^{\alpha, 1}$, where the boundedness condition is dropped, certain randomised Euler methods are introduced in \cite{Daun2011,Kruse2017,Heinrich2008} that converge with order $(\alpha+1/2)\wedge 1$, compared to the order $\alpha$ of Euler method. In particular, the key ingredients for the error analysis of the randomised Euler method in \cite{Kruse2017} (and Runge-Kutta method therein) is the randomised quadrature rule inspired by stratified Monte-Carlo method to approximate $\int_0^{j/n} g(r) \diff{r}$. The \emph{randomised quadrature} $Q_{\tau,n}^j[g]$ of
$\int_0^{j/n} g(r) \diff{r}$ with stepsize $1/n$ is given by 
\begin{align}
  \label{eq3:RandRie}
  Q_{\tau,n}^j[g] :=\int_0^{j/n} g(k^\tau_n(r))\,\de r = \frac{1}{n}\sum_{i = 0}^{j-1} g((i+ \tau_{i+1})/n ),\quad j \in
  \{1,\ldots,n\},
\end{align}
with $Q_{\tau,n}^0[g]=0$, where $k^\tau_n(r)=(\lfloor   nr \rfloor+\tau_{\lfloor   nr \rfloor})/n$ and $(\tau_i)_{i \in \N}$ is an independent family of
standard uniformly-distributed random variables on a probability space
$(\Omega_\tau,\F^\tau,\P_\tau)$. The important observation here is that the following error sequence with respect to $j\in \{0,1,\ldots, n\}$
\begin{equation}
    \int_0^{j/n} g(r) \diff{r}-\int_0^{j/n} g(k^\tau_n(r))\,\de r
\end{equation}
forms a discrete martingale such that the $L^p(\Omega_\tau)$-moments of the supremum of the error sequence can be bounded by the moments of its quadratic variation (see \cite[Theorem 3.1]{Kruse2017}). Thus, if $g$ is $\alpha$-H\"older continuous in time, one is able to gain an order $1/2+\alpha$ in $L^p(\Omega_\tau)$-sense, compared to an order $\alpha$ achieved by quantifying the error of \begin{equation}
    \int_0^{j/n} g(r) \diff{r}-\int_0^{j/n} g(k_n(r))\,\de r
\end{equation}
in the standard Euclidean norm.

In the stochastic setting, inspired by the randomised quadrature rule described above, both randomised Milstein method and randomised Euler-Maruyama method have been proposed to approximate solutions of different types of SDEs with time-irregular drift, including multiplicative SDEs \cite{Kruse2019a}, SDEs with jumps \cite{Przybylowicz2024}, stochastic delay differential equations \cite{Przybylowicz2024b} and McKean-Vlasov SDEs \cite{Biswas2024}. For the additive SDE \eqref{eq:SDE}, it is therefore evident to consider randomised EM method instead of EM method to achieve a higher order of convergence especially when $\alpha\in (0,1/2)$.

\subsection{The randomised Euler-Maruyama
method} 
For the introduction of the resulting \emph{randomised Euler-Maruyama
method},  let $(\tau_j)_{j \in \N}$ be an i.i.d.~family of
$\mathcal{U}(0,1)$-distributed random variables on an filtered
probability space $(\Omega_{\tau}, \F^{\tau}, (\F^{\tau}_j)_{j\in \N},
\P_{\tau})$, where $\F_j^{\tau}$ is the $\sigma$-algebra generated by
$\{\tau_1,\ldots, \tau_j\}$. Hereby, $\mathcal{U}(0,1)$ denotes the uniform 
distribution on the interval $(0,1)$. The random variables $(\tau_j)_{j \in
\N}$ represent the artificially added 
random input, which we assume
to be independent of the randomness already 
present in SDE \eqref{eq:SDE}. The resulting numerical method will then yield a discrete-time stochastic
process defined on the product probability space 
\begin{align}
  \label{eq:Omega}
  (\Omega, \F, \P) := (\Omega_{B}\times\Omega_{\tau}, \F^B\otimes \F^{\tau},
  \P_B \otimes \P_{\tau}).
\end{align}
Moreover, for each temporal stepsize $1/n$, $n\in \N$, the time grid for the numerical scheme is $\Pi_n:=\{t^n_j:=j/n, j= 0,1,\ldots, n\}$, and a 
discrete-time filtration $(\F^n_j)_{j \in \{0,\ldots,n\}}$
on $(\Omega, \F, \P)$ is defined as  
\begin{align}
  \label{eq:discretefiltration}
  \F^n_j := \F^B_{t^n_j} \otimes \F^{\tau}_{j+1}, \quad \text{ for } j \in
  \{0,1,\ldots,n\}.
\end{align}
We write $t_j$ instead of $t_j^n$ when there is no ambiguity. The randomised Euler-Maruyama approximation of SDE \eqref{eq:SDE} on $\Pi_n$ is given by
\begin{equation}\label{eqn:randomisedem}
    X^{(n)}_{t_j}=X^{(n)}_{t_{j-1}}+ f\big(t_{j-1}+\tau_j/n,X^{(n)}_{t_{j-1}}\big)/n+\Delta^n_j B,
\end{equation}
where $\Delta^n_j B:=B(t_j)-B(t_{j-1})$, with its continuous version 
\begin{equation}\label{eqn:randomisedemcontinuous}
    X^{(n)}_t=x_0+\int_0^t f\big(\kappa^\tau_n(s),X^{(n)}_{\kappa_n(s)}\big)\,\de s  +B_t,
\end{equation}
where we write $B(t)=B_t$ for short.
Note that continuous version of filtration is \begin{align}
  \label{eq:filtration}
  \F^n_t := \F^B_{t} \otimes \F^{\tau}_{\lfloor tn \rfloor+1}, \quad \text{ for } t \in
  (0,1].
\end{align}
Given the randomised scheme $X$ defined in \eqref{eqn:randomisedemcontinuous}, the key quadratic error terms $ \mathbb{I}_1$ and $ \mathbb{I}_2$ become
\begin{equation}\label{eqn:introterm1_random}
   \mathbb{I}^\tau_1(X^{(n)}):= \E\left[\sup_{s\leq u\leq t}\left|\int_{s}^u \left( f(r, X_r^{(n)}) - f(\kappa^\tau_n(r), X_{\kappa_n(r)}^{(n)}) \right)\,\de r\right|^p\right]
\end{equation}
and
\begin{equation}\label{eqn:introterm2_random}
    \mathbb{I}^\tau_2(X^{(n)},g):=\E\left[\sup_{s\leq u\leq t}\left|\int_{s}^u g(r,X^{\n}_{r}) \cdot  \left(f(r, X_{r}^{(n)})- f(\kappa^\tau_n(r), X_{\kappa_n(r)}^{(n)}) \right) \,\de r\right|^p\right].
\end{equation}
\subsection{Our contribution and organisation} With the randomised EM defined in \eqref{eqn:randomisedemcontinuous}, the strong order of convergence can be lifted to $1/2+\gamma-\epsilon$ for arbitrary $\epsilon\in (0,1/2)$, where $\gamma=\alpha \wedge (\beta/2)$, see Theorem \ref{thm:main1}. The error analysis follows the PDE strategy proposed in \cite{Pamen2017} but the key terms $\mathbb{I}^\tau_1$ and $\mathbb{I}^\tau_2$ are estimated through two quadratic bounds (Proposition \ref{lem:qb1} and Proposition \ref{lem:qb2}) using the stochastic sewing lemma. Our contributions are therefore two-fold:
\begin{itemize}
    \item the strong order of convergence of randomised EM (Theorem \ref{thm:main1}) is almost optimal (see Remark \ref{rmk:optimal});
    \item It is the first time the analysis of randomised schemes involving stochastic sewing lemma (see the proofs of Proposition \ref{lem:qb1}, Lemma \ref{lem:qb3_discrete} and Proposition \ref{lem:qb2}). 
\end{itemize}
For a comparison, we also include in Appendix \ref{sec:proof} an alternative proof for Lemma \ref{lem:qb3_discrete}, which achieves the same upper bound by constructing a discrete martingale from the process of quadrature errors and quantifying it through the discrete-time version of the Burkholder-Davis-Gundy inequality. 

The paper is structured as follows: Section \ref{sec:notation} outlines the standard notations and main tools to develop later error analysis, including the discrete-time version of the Burkholder-Davis-Gundy inequality (Theorem \ref{th:discreteBDG}) and the stochastic sewing lemma (Theorem \ref{thm:Sewing-lemma}). In Section \ref{sec:lemmas}, we detail the well-posedness of the numerical solution \eqref{eqn:randomisedemcontinuous} and derive the upper bounds for $\mathbb{I}^\tau_1$ and $\mathbb{I}^\tau_2$ (Proposition \ref{lem:qb1} and Proposition \ref{lem:qb2}).  Section \ref{sec:erroranalysis} is devoted to the main error analysis (Theorem \ref{thm:main1}).
\section{Notation and preliminaries}
\label{sec:notation}

In this section we explain the notation that is used throughout
this paper. In addition, we also collect a few notations and standard results from stochastic
analysis, which are needed in later sections.

By $\N$ we denote the set of all positive integers, while $\N_0 := \N \cup
\{0\}$. As usual the set $\R$ consists of all real numbers. By $| \cdot |$ we
denote the Euclidean norm on the Euclidean space $\R^d$ for any $d \in \N$. In
particular, if $d = 1$ then $| \cdot |$ coincides with taking
the absolute value. Moreover, the norm $| \cdot |_{\mathcal{L}(\R^d)}$ denotes
the standard matrix norm on $\R^{d \times d}$ induced by the Euclidean
norm. Set 
$$\nabla \equiv D =\begin{bmatrix}
    \frac{\partial}{\partial x_1}\\
    \ldots\\
    \frac{\partial}{\partial x_d} 
\end{bmatrix},$$ 
$D^2=(\frac{\partial^2}{\partial x_i x_j})_{1 \leq i,j \leq d}$ and $\Delta=\sum_{i=1}^d \frac{\partial^2}{\partial x_i^2}$. We use $a\wedge b$ and $a\vee b$ to denote the the minimum and maximum of $a$ and $b$.

In the following we introduce some space of function:
\begin{itemize}
	\item $C_b([s,t]\times \R^d;\R^m)$ with $s<t$, the set of all bounded functions from $[s,t]\times \R^d $ to $\R^m$, equipped with the norm
    $$\|g\|_\infty:=\sup_{t\in [0,1], x\in \R^d} |g(t,x)|.$$
	\item $C_b^{\beta}(\R^d;\R^m),\,\beta \in (0,1]$, the set of all functions from $\R^d $ to $\R^m$ which are bounded $\beta$-H\"older continuous functions.
	\item $C([s,t];C_b^{\beta}(\R^d;\R^d))$ with $s<t$, the collection of $C_b^{\beta}(\R^d;\R^d)$-valued continuous function over time interval $[s,t]$, equipped with the norm $\|\cdot\|_{C_b^{\beta}([s,t])}$ defined by
	\begin{align*}
	\|g\|_{C_b^{\beta}([s,t])} := \sup_{r \in [s,t], x \in \R^d}|g(r,x)|+\sup_{r \in [s,t], x \neq y}\frac{|g(r,x)-g(r,y)|}{|x-y|^{\beta}}.
	\end{align*}
     Note that $C[s,t];C_b^{\beta}(\R^d;\R^d))$ coincide with $\mathcal{C}_b^{0,\beta}([s,t]\times \mathbb{R}^d; \mathbb{R}^d)$ mentioned in Section \ref{sec:intro}, where the former one will be adopted in the later analysis.
    	\item $C^1([s,t];C_b^{\beta}(\R^d;\R^d))$ with $s<t$, the collection of $C_b^{\beta}(\R^d;\R^d)$-valued continuous function $g$ over time interval $[s,t]$, such that $\frac{\partial }{\partial r}g(r,\cdot)\in C_b^{\beta}(\R^d;\R^d)$ exists, and is continuous for $r\in [s,t]$.
	\item $C_b^{2,\beta}(\R^d)$ with $\beta \in (0,1)$ denotes the space of twice differentiable functions $g:\R^d \to \R$ with $D^{\ell}g \in C_b^{\beta}(\R^d; \R^{\otimes\ell})$ for any $1 \leq \ell \leq 2$. A function $g:\R^d \to \R^d$ belongs to $C_b^{i,\beta}(\R^d;\R^d)$ if each it components is in $C_b^{2,\beta}(\R^d)$ for $j=1,\ldots,d$.
    	\item $\mathcal{C}_b^{\alpha,\beta}([s,t]\times \mathbb{R}^d; \mathbb{R}^d)$ with $\alpha,\beta\in (0,1]$, the collection of  continuous functions $g:[s,t]\times \mathbb{R}^d\to \mathbb{R}^d$ which are $\alpha$-H\"older continuous in time and $\beta$-H\"older continuous in space, equipped with norm defined in Eqn. \eqref{eqn: fholder}. We will use $\mathcal{C}_b^{\alpha,\beta}([s,t]\times \mathbb{R}^d)$ for $\mathcal{C}_b^{\alpha,\beta}([s,t]\times \mathbb{R}^d; \mathbb{R})$ and use $\mathcal{C}_b^{\alpha,\beta}$ when it is no confusion.
    \item For an arbitrary Banach space $(E, \|\cdot\|_E)$ and for a given measure space $(X, \mathcal{A}, \mu)$ the set
$L^p(Y; E) := L^p( Y, \mathcal{A}, \mu; E)$, $p \in [1,\infty)$, consists of all
(equivalence classes of) Bochner measurable functions $g \colon Y \to E$ with
\begin{align*}
  \| g \|_{L^p(Y;E)} := \Big( \int_Y \|g(y)\|_E^p \diff{\mu(y)}
  \Big)^{\frac{1}{p}} < \infty.
\end{align*}
If $(E, \| \cdot\|_E ) = (\R, |\cdot|)$ we use the abbreviation $L^p(Y) :=
L^p(Y;\R)$. If $(Y, \mathcal{A}, \mu) = (\Omega, \F, \P)$ is a probability
space, we usually write the integral with respect to the probability measure
$\P$ as
\begin{align*}
  \E[ Z ] := \int_\Omega Z(\omega) \diff{\P(\omega)}, \quad Z \in
  L^p(\Omega;E).
\end{align*}
In the case of the product probability space $(\Omega, \F, \P)$ introduced in
\eqref{eq:Omega} an application of Fubini's theorem shows that
\begin{align*}
  \E[Z] = \EB[\E_{\tau}[ Z]] = \E_{\tau}[\EB[Z]], \quad Z \in
  L^p(\Omega;E),
\end{align*}
where $\EB$ is the expectation with respect to $\PB$ and $\E_{\tau}$ with 
respect to $\P_{\tau}$.
\end{itemize}

In the following, we will present two important tools, discrete-time version of the
Burkholder-Davis-Gundy inequality and stochastic sewing lemma. 

\begin{thm}\cite{burkholder1966} 
  \label{th:discreteBDG}
  For each $p \in (1,\infty)$ there exist positive constants $c_p$ and $C_p$
  such that for every discrete-time martingale $(Y^n)_{n \in \N_0}$ and
  for every $n \in \N_0$ we have 
  $$c_p \big\| [Y]_{n}^{\frac{1}{2}} \big\|_{L^p(\Omega)}
  \leq \big\| \max_{j\in \{0,\ldots,n\} } |Y^{j}| \big\|_{L^p(\Omega)} 
  \leq C_p \big\| [Y]_{n}^{\frac{1}{2}} \big\|_{L^p(\Omega)},$$
  where $[Y]_n = |Y^{0}|^2 + \sum_{k=1}^{n} |Y^{k}-Y^{k-1}|^2$
  is the \emph{quadratic variation} of $(Y^n)_{n \in \N_0}$.
\end{thm}

Define the simplex $\Delta_{S,T}:=\{(s,t)|S\leq s\leq t\leq T\}$.
\begin{thm} [{\cite[Theorem 2.4]{Le2020}}]           \label{thm:Sewing-lemma}
Consider a probability space $(\Omega, \mathcal{F},\{\mathcal{F}_t\}_{t\geq 0},\P)$. Let $p\geq 2$, $0\leq S\leq T$ and let $A_{\cdot,\cdot}$ be a function $\Delta_{S,T}\to L^p(\Omega;\R^d)$ such that for any $(s,t)\in\Delta_{S,T}$ the random vector $A_{s,t}$ is $\F_t$-measurable. Suppose that for some $\epsilon_1,\epsilon_2>0$ and $C_1,C_2$ the bounds
\begin{align}
   & \|A_{s,t}\|_{L^p(\Omega;\R^d)}  \leq C_1|t-s|^{1/2+\epsilon_1},\label{SSL1}
\\
&\|\E[\delta A_{s,u,t}|\F_s]\|_{L^p(\Omega;\R^d)}=:\|\E^s[\delta A_{s,u,t}]\|_{L^p(\Omega;\R^d)}  \leq C_2 |t-s|^{1+\epsilon_2},\label{SSL2}
\end{align}
hold for all $S\leq s\leq u\leq t\leq T$, where
$$\delta A_{s,u,t}:=A_{s,t}-A_{s,u}-A_{u,t}.$$
Then there exists a unique (up to modification) $\{\F_t\}$-adapted 
process $\A:[S,T]\to L^p(\Omega;\R^d)$ such that $\A_S=0$ and the following bounds hold for some positive constants $K_1,K_2$:
\begin{align}
&\|\A_t	-\A_s-A_{s,t}\|_{L^p(\Omega;\R^d)}  \leq K_1 |t-s|^{1/2+\epsilon_1}+K_2 |t-s|^{1+\epsilon_2},\quad (s,t)\in\Delta_{S,T},\label{SSL1 cA}
\\
&\|\E^s\big[\A_t	-\A_s-A_{s,t}\big]\|_{L^p(\Omega;\R^d)}  \leq K_2|t-s|^{1+\epsilon_2},\quad (s,t)\in\Delta_{S,T}\label{SSL2 cA}.
\end{align}
Moreover, there exists a positive constant $K$ depending only on $\epsilon_1,\epsilon_2$ and $d$ such that $\A$
satisfies the bound
\begin{equation}\label{SSL3 cA}
\|\A_t-\A_s\|_{L^p(\Omega;\R^d)}  \leq  KpC_1 |t-s|^{1/2+\epsilon_1}+KpC_2 |t-s|^{1+\epsilon_2},\quad (s,t)\in[S,T].
\end{equation}
\end{thm}

\section{Quadratic estimates} \label{sec:lemmas}
\begin{lem}[Well-posedness]\label{lem:Xnestimate}
Consider the randomised numerical simulation \eqref{eqn:randomisedemcontinuous}. Suppose that $f$ is bounded, i.e., $f \in C_b([0,1]\times \R^d;\R^d)$. Then $(X^\n_t)_{0\leq t\leq 1}$ is well-defined, in the sense that for any $p> 0$, 
$$ \sup_{0\leq t\leq 1}\|X^\n_t\|_{L^p(\Omega; \R^d)}<\infty.$$

\end{lem}

\begin{proof}Define $Y^\n_t:=X^\n_t-B(t)$. Then \eqref{eqn:randomisedemcontinuous} becomes
$$Y^\n_t=x_0+\int_0^t f(s,Y^\n_s+B(s))\,\de s$$
    It holds from the boundedness of $f$ that $|Y^\n_t|\leq |x_0|+t\|f\|_\infty.$ Therefore,
    \begin{align*}
      \|X^\n_t\|_{L^p(\Omega; \R^d)}&\leq \|Y^\n_t\|_{L^p(\Omega; \R^d)}  +\|B(t)\|_{L^p(\Omega; \R^d)} <\infty.
    \end{align*}
\end{proof}

\begin{prop}[Quadratic bound 1]\label{lem:qb1}
Let $\alpha,\beta\in(0,1]$, $p>0$, $d,m\in \N$, and take $\epsilon\in(0,1/2)$.
Then
for all $g\in \Cb([0,1]\times \mathbb{R}^d;\mathbb{R}^d)$,
$0\leq s\leq t\leq 1$, $n\in\N$, there exists an constant $\bar{C}(p, d,\beta,\epsilon)$ such that
\begin{align}\label{eqn:qb1}
    \begin{split}
  &\Bigl\|\int_s^t (g(r,B_r)-g(\kappa_n^\tau(r),B_{\kappa_n(r)}))\, \de r\Bigr\|_{L^p(\Omega;\R^d)}\\
  &\leq \bar{C}(p, d,\beta,\epsilon)\|g\|_{\Cb} |t-s|^{1/2+\epsilon}n^{-(1/2+\gamma-\epsilon)},     
    \end{split}
\end{align}
where $\gamma=\alpha \wedge (\beta /2)$.
\end{prop}
\begin{proof}
It suffices to prove the bound for $p\geq 2$.
Define for $0\leq s\leq t\leq 1$
$$
A_{s,t}:=\E^s\left[ \int_s^t (g(r,B_r)-g(\kappa^\tau_n(r),B_{\kappa_n(r)}))\, \de r\right].
$$
For any $0\leq s\leq u\leq t\leq 1$
\begin{align*}
\delta A_{s,u,t}&=A_{s,t}-A_{s,u}-A_{u,t}\\
&=\E^s \left[\int_u^t (g(r,B_r)-g(\kappa_n^\tau(r),B_{\kappa_n(r)}))\, \de r\right]\\
&\quad-\E^u \left[\int_u^t(g(r,B_r)-g(\kappa^\tau_n(r),B_{\kappa_n(r)}))\, \de r\right].
\end{align*}
Let us check that all the conditions of the stochastic sewing lemma (Theorem~\ref{thm:Sewing-lemma}) are satisfied.

The condition \eqref{SSL2} trivially holds with $C_2=0$ because of 
\begin{equation*}
\E^s [\delta A_{s,u,t}]=0,
\end{equation*}
by the property of conditional expectation.

Denote $k_2:=\lfloor tn\rfloor \in \N$.  To establish \eqref{SSL1}, let $s \in [k/n, (k+1)/n)$ for some $k \in \{0,\hdots,k_2-1\}$.
Suppose first that $t \in [(k+2)/n, 1]$ so that $(k+2)/n-s>n^{-1}$. We write
\begin{align}\label{eqn:Astdec}
    \begin{split}
 |A_{s,t}|&\leq  \int_s^{(k+2)/n}
  \left|\E^s \big[ g(r,B_r)-g(\kappa_n^\tau(r),B_{\kappa_n(r)})\big]\right|\, \de r\\
  &\quad +\left|\E^s \left[\int^t_{(k+2)/n}
  \big(g(r,B_{\kappa_n(r)})-g(\kappa_n^\tau(r),B_{\kappa_n(r)})\big)\, \de r \right] \right|\\
    &\quad + \int^t_{(k+2)/n}
  \left|\E^s \big[ g(r,B_r)-g(r,B_{\kappa_n(r)})\big]\right|\, \de r\\
  &=:I_1+I_2+I_3.       
    \end{split}
\end{align}

The bound for $I_1$ is straightforward: by conditional Jensen's inequality and the first estimate of Proposition \ref{prop:fractional} we have
\begin{align}\label{eqn:AstI1}
\begin{split}
  \| I_1\|_{L^p(\Omega; \R^d)} &\leq
\int_s^{(k+2)/n} \| g(r,B_r)-g(\kappa^\tau_n(r),B_{\kappa_n(r)}) \|_{L^p(\Omega;\R^d)} \, \de r
\\
&\leq 2\bar{C}(p,d,\beta)
\|g\|_{\Cb} (1/n)^{1+(\alpha \wedge (\beta /2))} \\
&\leq 2\bar{C}(p,d,\beta)\|g\|_{\Cb} (1/n)^{(1/2+(\alpha \wedge (\beta /2))-\epsilon}|t-s|^{1/2+\epsilon},  
\end{split}
\end{align}
where the last inequality follows from the fact that $1/n\leq|t-s|$.

Regarding $I_2$, note that for each $k_1\in \{k+2,\ldots, k_2-1\}$,
\begin{align}\label{eqn:randomtrick}
\begin{split}
        &\E^s \left[\int^{(k_1+1)/n}_{k_1/n}
  \big(g(r,B_{\kappa_n(r)})-g(\kappa_n^\tau(r),B_{\kappa_n(r)})\big)\, \de r\right]\\
  &=\E^s_B \left[\E^{s}_\tau\left[\int^{(k_1+1)/n}_{k_1/n}
  g(r,B_{k_1/n})\, \de r-\frac{1}{n}f\big((k_1+\tau_{1+k_1})/n,B_{k_1/n}\big)\right]\right]\\
&=\E^s_B \left[\int^{(k_1+1)/n}_{k_1/n}
  \big(g(r,B_{k_1/n})\, \de r-\frac{1}{n}\E_\tau^{s}\left[f\big((k_1+\tau_{1+k_1})/n,B_{k_1/n}\big)\right]\right]=0.
\end{split}
\end{align}
Thus
\begin{align*}
 I_2&=\left|\E^s \left[\int^t_{k_2/n}
  \big(g(r,B_{\kappa_n(r)})-g(\kappa_n^\tau(r),B_{\kappa_n(r)})\big)\, \de r \right] \right|  \\
  &\leq \int^t_{k_2/n}
  \left|\E^s \big[ g(r,B_{\kappa_n(r)})-g(\kappa_n^\tau(r),B_{\kappa_n(r)})\big]\right|\, \de r.
\end{align*}
Similar to the estimate in \eqref{eqn:AstI1},
\begin{align}\label{eqn:AstI2}
    \begin{split}
    \| I_2\|_{L^p(\Omega;\R^d)}  
&\leq \|g\|_{\Cb} n^{-(1+\alpha) }\leq \|g\|_{\Cb} n^{-(1/2+\alpha-\epsilon)}|t-s|^{1/2+\epsilon}.
    \end{split}
\end{align}

Regarding $I_3$, the argument is exactly the same as the proof of \cite[Lemma 4.3]{Butkovsky2021} . Using the first estimate of Proposition \ref{prop:fractional}, we derive
\begin{align}\label{eqn:AstI3}
\begin{split}
I_3\le&\int_{(k+2)/n}^t|\cP_{|r-s|}g(r,\E^s[B_r])-\cP_{|\kappa_n(r)-s|}g(r,\E^s[B_r])|\,\de r
\\
&+\int_{(k+2)/n}^t \left|\cP_{|\kappa_n(r)-s|}\left(g(r,\E^s[B_r])-g(r,\E^s[B_{\kappa_n(r)}])\right)\right|\,\de r\\
=:&I_{31}+I_{32}.   
\end{split}
\end{align}
To bound $I_{31}$, we apply
estimate \eqref{eqn:A2difference} of Proposition \ref{prop:HK}  with $\eta=0$, $\delta=1$ and the fourth estimate of Proposition \ref{prop:fractional}. We get
\begin{align}\label{eqn:AstI31}
\begin{split}
    &\|I_{31}\|_{L^p(\Omega;\R^d)}\\
    &\leq \bar{C}(d,p,\beta)\| g\|_{C^{\alpha,\beta}} \int_{(k+2)/n}^t\big(|r-s|-|\kappa_n(r)-s|\big)|\kappa_n(r)-s|^{\beta/2-1}\,\de r\\
&\leq\bar{C}(d,p,\beta)\| g\|_{\Cb} n^{-1}\int_{(k+2)/n}^t|r-1/n-s|^{\beta /2-1}\,\de r\\
&=\bar{C}(d,p,\beta)\| g\|_{\Cb} n^{-1}\int_{(k+1)/n-s}^{t-1/n-s}|\bar{r}|^{\beta /2-1}\,\de \bar{r}\\
&\leq 2\beta^{-1}n^{-1}\bar{C}(d,p,\beta)\| g\|_{\Cb}|t-s|^{\beta /2},
\end{split}
\end{align}
where $(k+1)/n-s>0$, making the last second integral well-defined.

Regarding $I_{32}$,  we use \eqref{eqn:A2single} of Proposition~\ref{prop:HK} with $\eta=1$ and the last estimate of Proposition~\ref{prop:fractional}. We deduce
\begin{align}\label{eqn:AstI32}
\begin{split}
    &\|I_{32}\|_{L^p(\Omega;\R^d)}\\
    &
\leq \bar{C}(d,p,\beta)\| g\|_{\Cb}\int_{(k+2)/n}^t \|\E^s[B_r]-\E^s[B_{\kappa_n(r)}]\|_{L^p(\Omega)}|\kappa_n(r)-s|^{\beta/2-1/2}\,\de r
\\
&\leq 2\beta^{-1}n^{-1}\bar{C}(d,p,\beta)\| g\|_{\Cb}|t-s|^{\beta /2}.
\end{split}
\end{align}
Combining \eqref{eqn:AstI31} and \eqref{eqn:AstI32}, and taking again into account that $n^{-1}\leq|t-s|$, we get
\begin{equation}
   \| I_3\|_{L^p(\Omega;\R^d)} \leq \bar{C}(d,p,\beta)\| g\|_{C^{\alpha,\beta}} n^{-(1/2+\beta /2-\epsilon)}|t-s|^{1/2+\epsilon}. 
\end{equation}
The case of $t\in [k/n,(k+2)/n)$ can be dealt with easily so we omit it here.

Recalling \eqref{eqn:AstI1} and \eqref{eqn:AstI2}, we finally conclude
\begin{equation}
\| A_{s,t}\|_{L^p(\Omega;\R^d)}\leq \bar{C}(d,p,\beta)\| g\|_{C^{\alpha,\beta}} n^{-(1/2+\gamma-\epsilon)}|t-s|^{1/2+\epsilon}.
\end{equation}

Thus all the conditions of the stochastic sewing lemma are satisfied.
The process
$$
\tilde \A_t:=\int_0^t  (g(r,B_r)-g(\kappa^\tau_n(r),B_{\kappa_n(r)}))\,\de r
$$
is also $\F_t$-adapted, satisfies \eqref{SSL2 cA} trivially, and
$$
\| \tilde{\mathcal{A}}_t-\tilde{\mathcal{A}}_s-A_{s,t}\|_{L^p(\Omega;\R^d)} \leq \bar{C}(d,p,\beta) |t-s|^{1/2+\epsilon},
$$
which shows that it also satisfies \eqref{SSL1 cA}.
Therefore by uniqueness $\A_t=\tilde \A_t$, the bound \eqref{SSL3 cA} then yields precisely \eqref{eqn:qb1}.
\end{proof}
\begin{lem}
    [Quadratic bound 2]\label{lem:qb3_discrete}
Let $\alpha,\beta\in(0,1]$, $p>0$, and take $\epsilon\in (0,1/2)$.
Then
for all $g_1\in \mathcal{C}_b^{\alpha,\beta}([0,1]\times \R^d)$ and $g_2\in C_b([0,1]\times \R^d)$,
$0\leq s\leq t\leq 1$, $n\in\N$, there exists an constant $\bar{C}(p, \epsilon)$ such that
\begin{align}\label{eqn:errRie}
    \begin{split}
  &\Bigl\|\int_s^t \big(g_1(r,B_{\kappa_n(r)})-g_1(\kappa_n^\tau(r),B_{\kappa_n(r)})\big)g_2(\kappa_n(r),B_{\kappa_n(r)})\, \de r\Bigr\|_{L^p(\Omega)}\\
  &\leq   \bar{C}(p, \epsilon)\|g_1\|_{\mathcal{C}^{\alpha,\beta}_b}\|g_2\|_{\infty} |t-s|^{1/2+\epsilon}n^{-(1/2+\alpha-\epsilon)}. 
    \end{split}
\end{align}
\end{lem}
\begin{proof}

 Define for $s,t$ with $0\leq s\leq t\leq 1$
\begin{align*}
    \tilde{\mathcal{A}}_{t}:=\int_0^t \big(g_1(r,B_{\kappa_n(r)})-g_1(\kappa_n^\tau(r),B_{\kappa_n(r)})\big)g_2(\kappa_n(r),B_{\kappa_n(r)})\, \de r
\end{align*}
and 
\begin{align*}
    A_{s,t}:=\E^s\Big[\int_s^t \big(g_1(r,B_{\kappa_n(r)})-g_1(\kappa_n^\tau(r),B_{\kappa_n(r)})\big)g_2(\kappa_n(r),B_{\kappa_n(r)})\, \de r\Big],
\end{align*}
Similar to the proof of Proposition \ref{lem:qb1}, it can be easily verified that $\E^s [\delta A_{s,u,t}]=0$ and $\E^s [\tilde{\mathcal{A}}_{t}-\tilde{\mathcal{A}}_{s}-A_{s,t}]=0$. Thus conditions \eqref{SSL2} and \eqref{SSL2 cA} are satisfied. It remains to check conditions \eqref{SSL1} and \eqref{SSL1 cA}. 

 When $t-s < n^{-1}$, it is easy to get
\begin{align}\label{eqn:firstcase}
\begin{split}
     &\Bigl\|A_{s,t}\|_{L^p(\Omega)}\leq \Bigl\|\tilde{\mathcal{A}}_{t}-\tilde{\mathcal{A}}_{s}\|_{L^p(\Omega)}\\
  &\leq   \|g_1\|_{\mathcal{C}^{\alpha,\beta}_b}\|g_2\|_{\infty} |t-s|n^{-\alpha}\\
  &\leq \|g_1\|_{\mathcal{C}^{\alpha,\beta}_b}\|g_2\|_{\infty} |t-s|^{1/2+\epsilon}n^{-(1/2+\alpha-\epsilon)}.  
\end{split}
\end{align}
Now let us assume $t-s\geq n^{-1}$, and let $s \in [k_1/n, (k_1+1)/n)$ and $t \in [k_2/n, (k_2+1)/n)$ for some $k_1,k_2 \in \{0,\ldots, n-1\}$. Applying the same argument as in \eqref{eqn:randomtrick} yields that
\begin{align*}
    A_{s,t}=&\E^s\Big[\int_s^{(k_1+1)/n}\big(g_1(r,B_{\kappa_n(r)})-g_1(\kappa_n^\tau(r),B_{\kappa_n(r)})\big)g_2(\kappa_n(r),B_{\kappa_n(r)})\, \de r\Big]\\
    &+\E^s\Big[\int_{k_2}^{t}\big(g_1(r,B_{\kappa_n(r)})-g_1(\kappa_n^\tau(r),B_{\kappa_n(r)})\big)g_2(\kappa_n(r),B_{\kappa_n(r)})\, \de r\Big].
\end{align*}
As $(k_1+1)/n -s <n^{-1}$ and $t-k_2/n <n^{-1}$, each of the two terms falls into the case considered in \eqref{eqn:firstcase}, thus is bounded by $\|g_1\|_{\mathcal{C}^{\alpha,\beta}_b}\|g_2\|_{\infty} |t-s|^{1/2+\epsilon}n^{-(1/2+\alpha-\epsilon)}$ in the $L^p$ norm. By now condition \eqref{SSL1} has been verified.

Finally, let us verify that $ \tilde{\mathcal{A}}_{\cdot}$ is the sewing of $A_{\cdot,\cdot}$ by checking the condition \eqref{SSL1 cA}. Indeed, condition \eqref{SSL1 cA} is satisfied given the condition \eqref{SSL1} and the second line of \eqref{eqn:firstcase}.
\end{proof}
An alternative proof of Lemma \ref{lem:qb3_discrete} directly using the martingale argument is presented in Appendix \ref{sec:proof}.
\begin{prop}
    [Quadratic bound 3]\label{lem:qb2}
Let $\alpha,\beta\in(0,1]$, $p>0$, and take $\epsilon\in (0,1/2 )$.
Then
for all $g_1\in \mathcal{C}_b^{\alpha,\beta}([0,1]\times \R^d)$ and $g_2\in C([0,1]; C^1_b(\R^d))$,
$0\leq s\leq t\leq 1$, $n\in\N$, there exists an constant $\bar{C}(p, d,\beta,\epsilon)$ such that
\begin{align}\label{eqn:qb2}
    \begin{split}
  &\Bigl\|\int_s^t (g_1(r,B_r)-g_1(\kappa_n^\tau(r),B_{\kappa_n(r)}))g_2(r,B_{r})\, \de r\Bigr\|_{L^p(\Omega)}\\
  &\leq \bar{C}(p, d,\beta,\epsilon)\|g_1\|_{\mathcal{C}^{\alpha,\beta}_b}\|g_2\|_{C_b^1([0,1])} |t-s|^{1/2+\epsilon}n^{-(1/2+\gamma-\epsilon)},     
    \end{split}
\end{align}
where $\gamma=\alpha \wedge (\beta/2)$.
\end{prop}
\begin{proof}
Define and divide the process as follows:
\begin{align*}
    \tilde{\mathcal{A}}_{t}:=\int_0^t (g_1(r,B_r)-g_1(\kappa_n^\tau(r),B_{\kappa_n(r)}))g_2(r,B_{r})\, \de r:= \tilde{\mathcal{A}}_{t}^1 + \tilde{\mathcal{A}}_{t}^2,
\end{align*}
where
$$
\tilde \A_t^1:=\int_0^t  (g_1(r,B_r)-g_1(r,B_{\kappa_n(r)}))g_2(r,B_{r})\,\de r
$$
and
$$
\tilde \A_t^2:=\int_0^t  (g_1(r,B_{\kappa_n(r)})-g_1(\kappa_n^\tau(r),B_{\kappa_n(r)}))g_2(r,B_{r})\,\de r
$$
are two $\F_t$-adapted processes.

Define for $0\leq s\leq t\leq  1$
$$
A_{s,t}^1:=\E^s\left[ \int_s^t (g_1(r,B_r)-g_1(r,B_{\kappa_n(r)}))g_2(0,B_{s})\, \de r\right],
$$
and
$$
A_{s,t}^2:=\int_s^t (g_1(r,B_{\kappa_n(r)})-g_1(\kappa_n^\tau(r),B_{\kappa_n(r)}))g_2({\kappa_n(r)},B_{\kappa_n(r)})\, \de r.
$$
First, let us claim that $\tilde \A_{\cdot}^1$ (resp. $\tilde \A_{\cdot}^2$) is the sewing of $A_{{\cdot},{\cdot}}^1$ (resp. $A_{{\cdot},{\cdot}}^2$) by checking conditions \eqref{SSL1 cA} and \eqref{SSL2 cA}. These can be easily verified as follows:
\begin{align*}
  & \| \tilde{\mathcal{A}}_t^1-\tilde{\mathcal{A}}_s^1-A_{s,t}^1\|_{L^p(\Omega)}  \\
   &= \left\| \E^s\left [\int_s^t  \big(g_1(r,B_r)-g_1(r,B_{\kappa_n(r)})\big)\left(g_2(r,B_{r})-g_2(0,B_{s})\right)\, \de r\right]\right\|_{L^p(\Omega)} \\
   &\leq \bar{C}(p,d) \|g_1\|_{\mathcal{C}^{\alpha,\beta}_b}\|g_2\|_{\infty} |t-s|^{1+\beta/2}.
\end{align*}
and
\begin{align*}
  & \| \tilde{\mathcal{A}}_t^2-\tilde{\mathcal{A}}_s^2-A_{s,t}^2\|_{L^p(\Omega)}  \\
   &\leq \left\| \int_s^t  \big(g_1(r,B_r)-g_1(\kappa_n^\tau(r),B_{\kappa_n(r)})\big)\big(g_2(r,B_{r})-g_2(\kappa_n(r),B_{\kappa_n(r)})\big)\,\de r\right\|_{L^p(\Omega)} \\
   &\leq \bar{C}(p,d) \|g_1\|_{\mathcal{C}^{\alpha,\beta}_b}\|g_2\|_{\infty} |t-s|^{1+\gamma}.
\end{align*}
Note that the condition \eqref{SSL2 cA} is automatically satisfied using the estimates above and Jensen's inequality for conditional expectation.

Next, let us check that the remaining conditions \eqref{SSL1} and \eqref{SSL2} of the stochastic sewing lemma (Theorem~\ref{thm:Sewing-lemma}) are satisfied for both $A_{s,t}^1$ and $A_{s,t}^2$. Because of the fact $\E^s [\delta A^2_{s,u,t}]=0$ and Lemma \ref{lem:qb3_discrete}, both the conditions are valid for $A_{s,t}^2$. It remains to check the conditions for $A_{s,t}^1$.
It is clearly that
\begin{equation*}
\E^s [\delta A^1_{s,u,t}]=\E^s \left[\E^u \left[\int_u^t\,g_1(r,B_r)-g_1(r,B_{\kappa_n(r)}) \de r \right]\left(g_2(0,B_{u})-g_2(0,B_{s}) \right) \right].
\end{equation*}

To estimate it,
\begin{align*}
   &\| \E^s [\delta A^1_{s,u,t}]\|_{L^p(\Omega)}\\
   &\leq  \left\| \E^u \left[\int_u^t g_1(r,B_r)-g_1(r,B_{\kappa_n(r)}) \,\de r\right] \right\|_{L^{2p}(\Omega)}  \\
  & \qquad \qquad \times \left \| g_2(0,B_{u})-g_2(0,B_{s})\right\|_{L^{2p}(\Omega)}\\
  &\leq  \bar{C}(2p,d)\|g_2\|_{\mathcal{C}^{0,1}_b} |u-s|^{1/2} \bar{C}(2p,d,\beta,\epsilon)\|g_1\|_{\mathcal{C}^{\alpha,\beta}_b} |t-u|^{1/2+\epsilon}n^{-(1/2+\gamma-\epsilon)}\\
  &\leq \bar{C}(2p,d,\beta,\epsilon)|t-s|^{1/2+\epsilon+1/2} \|g_2\|_{C^{1}_b([0,1])} \|g_1\|_{\mathcal{C}^{\alpha,\beta}_b} n^{-(1/2+\gamma-\epsilon)}\\
    &= \bar{C}(2p,d,\beta,\epsilon)|t-s|^{1+\epsilon} \|g_2\|_{C^{1}_b([0,1])} \|g_1\|_{\mathcal{C}^{\alpha,\beta}_b} n^{-(1/2+\gamma-\epsilon)},
\end{align*}
where we apply Proposition \ref{lem:qb1} with $\epsilon\in (0,1/2)$ to derive the last inequality
and so condition \eqref{SSL2} holds.

Similarly,

\begin{align*}
   &\| A_{s,t}^1\|_{L^p(\Omega)}\\
   &\leq  \left\| \E^s \left[\int_s^t g_1(r,B_r)-g_1(r,B_{\kappa_n(r)}) \,\de r\right] \right\|_{L^{2p}(\Omega)} \left \| g_2(s,B_{s})\right\|_{L^{2p}(\Omega)}\\
  &\leq \bar{C}(2p,d,\beta,\epsilon)|t-s|^{1/2+\epsilon} \|g_2\|_{\infty} \|g_1\|_{\mathcal{C}^{\alpha,\beta}_b} n^{-(1/2+\gamma-\epsilon)}.
\end{align*}
Thus all the conditions of the stochastic sewing lemma are satisfied.
\end{proof}
\begin{rmk}\label{rmk:shift}
    The quadratic bounds shown in Proposition \ref{lem:qb1} and Proposition \ref{lem:qb2} can be easily extended to the case when the Brownian motion $B$ is shifted by a constant $x\in \R^d$, i.e.,  the estimates \eqref{eqn:qb1} and \eqref{eqn:qb2} still hold if one replaces $B$ with $B+x$.
\end{rmk}
\begin{rmk}
    The quadratic bounds shown in Proposition \ref{lem:qb1} and Proposition \ref{lem:qb2} can be easily extended to the case of fractional Brownian motion (fBM) using \cite[Proposition 3.6, Proposition 3.7]{Butkovsky2021}. 
\end{rmk}

\begin{lem}\label{lem:girsonovstrans}
Let $p>0$, $\epsilon\in(0,1/2)$, and $X^\n$ be the solution of \eqref{eqn:randomisedemcontinuous} with $f\in \Cb([0,1]\times\R^d;\R^d)$.
Then for all $g\in\Cb([0,1]\times\R^d;\R^d)$, $0\leq s\leq t \leq 1$, there exists a constant $\bar{C}(p, d,\beta,\eps)$ such that
\begin{align}\label{eqn:qb1x}
    \begin{split}
   & \Bigl\|\int_s^t (g(r,X_r^\n)-g(\kappa^\tau_n(r),X^\n_{\kappa_n(r)}))\, \de r\Bigr\|_{L^p(\Omega;\R^d)}\\
&\leq \bar{C}(p, d,\beta,\epsilon)\|f\|_{\Cb}\|g\|_{\Cb} |t-s|^{1/2+\epsilon} n^{-(1/2+\gamma-\epsilon)},    
    \end{split}
\end{align}
where $\gamma=\alpha \wedge (\beta /2)$.
Moreover, for all $g_1\in \mathcal{C}_b^{\alpha,\beta}([0,1]\times \R^d)$ and $g_2\in C([0,1];C^1_b( \R^d))$,
$0\leq s\leq t\leq 1$, $n\in\N$, there exists an constant $\bar{C}(p, d,\beta,\epsilon)$ such that
\begin{align}\label{eqn:qb2x}
    \begin{split}
  &\Bigl\|\int_s^t (g_1(r,X_r)-g_1(\kappa_n^\tau(r),X_{\kappa_n(r)}))g_2(r,X_{r})\, \de r\Bigr\|_{L^p(\Omega)}\\
  &\leq \bar{C}(p, d,\beta,\epsilon)\|f\|_{\Cb}\|g_1\|_{\mathcal{C}^{\alpha,\beta}_b}\|g_2\|_{C^{1}_b([0,1])} |t-s|^{1/2+\epsilon}n^{-(1/2+\gamma-\epsilon)}.     
    \end{split}
\end{align}
\end{lem}

\begin{proof}
Fix an arbitrary realisation $\omega_\tau\in \Omega_\tau$. Let
$$
\psi^n(t,\omega_\tau):=\int_0^t f(\kappa^\tau_n(r,\omega_\tau),X^\n_{\kappa_n(r)}(\omega_\tau))\,\de r.
$$
Take $v(t,\omega_\tau)=f\big(\kappa^\tau_n(t,\omega_\tau),X^\n_{\kappa_n(t)}(\omega_\tau)\big)$. 
Note that for each $\omega_\tau$
\begin{equation}\label{eqn:girsonov1}
   \EB\!\left[\exp\Big(\frac{1}{2}\int_{0}^{1} v^2(t,\omega_\tau) \,\de t \Big)\right]\leq \exp\big(\|f\|^2_{\Cb}/2\big).
\end{equation}
 Let us apply the Girsanov theorem to the function $v$, then there exists a probability measure $\tilde\PB$ equivalent to $\PB$ such that the process $\tilde B:=B+\psi^n(\omega_\tau)$ is a
Brownian motion on $[0,1]$ under $\tilde\PB$.

In the following, we will use $\EB[\cdot |\mathcal{F}^\tau]$ or $\E_{\tilde B}[\cdot |\mathcal{F}^\tau]$ to emphasise the expectation is evaluated under probability $\PB$ or $\tilde \PB$, given a fixed realisation $\omega_\tau$, without directly mentioning $\omega_\tau$. We deduce by \eqref{eqn:girsonov1} that
\begin{equation*}
\EB\Bigl[\frac{\de\PB}{\de\tilde\PB}\,\Big\vert\, \F^\tau\Bigr]\le \exp\big(\|f\|^2_{\Cb}/2\big).
\end{equation*}
Moreover, for each $p$, we have
\begin{align}\label{eqn:girsanov2}
\begin{split}
  &\EB \!\left[\Bigl|  \int_s^t \left( g(r, X^\n_r)- g(\kappa^\tau_n(r),X^\n_{\kappa_n(r)}) \right) \, \de r \Bigr|^p \,\Big\vert\, \F^\tau\right]\\
&=\E_{\tilde B}\! \left[\Bigl|  \int_s^t \left( g(r,X^\n_r)- g(\kappa^\tau_n(r),X^\n_{\kappa_n(r)}) \right) \, \de r \Bigr|^p\frac{\de\PB}{\de\tilde\PB} \,\Big\vert\, \F^\tau\right]\\
&\le \left(\E_{\tilde B}\! \left[\Bigl|  \int_s^t \left( g(r,X^\n_r)- g(\kappa^\tau_n(r),X^\n_{\kappa_n(r)}) \right) \, \de r \Bigr|^{2p}\,\Big\vert\, \F^\tau\right]\right)^{1/2}\Bigl(\E_{\tilde B}\Bigl[\left(\frac{\de\PB}{\de\tilde\PB}\right)^2\,\Big\vert\, \F^\tau\Bigr]\Bigr)^{1/2}
\\
&=\left(\EB\! \left[\Bigl|  \int_s^t \left( g(r, B_r+x_0^n)- g(\kappa^\tau_n(r), B_{\kappa_n(r)}+x_0^n) \right) \, \de r \Bigr|^{2p}\,\Big\vert\, \F^\tau\right]\right)^{1/2}\Bigl(\EB\Bigl[\frac{\de\PB}{\de\tilde\PB}\,\Big\vert\, \F^\tau\Bigr]\Bigr)^{1/2}\\
&\leq \|f\|_{\Cb}\left(\EB\! \left[\Bigl|  \int_s^t \left( g(r, B_r+x_0^n)- g(\kappa^\tau_n(r), B_{\kappa_n(r)}+x_0^n) \right) \, \de r \Bigr|^{2p}\,\Big\vert\, \F^\tau\right]\right)^{1/2}.
\end{split}
\end{align}

Taking expectation with respect to $\omega_\tau$ on both side of \eqref{eqn:girsanov2} and using \eqref{eqn:qb1} and Remark \ref{rmk:shift} yield:
\begin{align*}
&\E \!\left[\Bigl|  \int_s^t \left( g(r, X^\n_r)- g(\kappa^\tau_n(r),X^\n_{\kappa_n(r)}) \right) \, \de r \Bigr|^p \right]\\
&\le \bar{C}(2p,d,\beta,\epsilon)\|f\|^p_{\Cb}\|g\|_{\Cb}^p n^{-p(1/2+\gamma-\epsilon)}|t-s|^{p(1/2+\epsilon)}.
\end{align*}
A similar argument leads to Eqn. \eqref{eqn:qb2x}.

\end{proof}

Applying the Kolmogorov continuity theorem to Lemma \ref{lem:girsonovstrans} yields
\begin{cor}\label{cor:qbsup}
Let $p>0$, $\epsilon\in(0,1/2)$, and $X^\n$ be the solution of \eqref{eqn:randomisedemcontinuous} with $f\in \Cb([0,1]\times\R^d;\R^d)$.
Then for all $g\in\Cb([0,1]\times\R^d;\R^d)$, $0\leq t \leq 1$, there exists a constant $\bar{C}(p, d,\beta,\eps)$ such that
\begin{align}\label{eqn:qb1sup}
    \begin{split}
   & \Bigl\|\sup_{0\leq s\leq t}\big|\int_0^s (g(r,X_r^\n)-g(\kappa^\tau_n(r),X^\n_{\kappa_n(r)}))\, \de r \big|\Bigr\|_{L^p(\Omega;\R^d)}\\
&\leq \bar{C}(p, d,\beta,\epsilon)\|f\|_{\Cb}\|g\|_{\Cb}  n^{-(1/2+\gamma-\epsilon)},    
    \end{split}
\end{align}
where $\gamma=\alpha \wedge (\beta /2)$.
Moreover, for all $g_1\in \mathcal{C}_b^{\alpha,\beta}([0,1]\times \R^d)$ and $g_2\in C([0,1];C^1_b( \R^d))$,
$0\leq t\leq 1$, $n\in\N$, there exists an constant $\bar{C}(p, d,\beta,\epsilon)$ such that
\begin{align}\label{eqn:qb2sup}
    \begin{split}
  &\Bigl\|\sup_{0\leq s\leq t}\big| (g_1(r,X_r)-g_1(\kappa_n^\tau(r),X_{\kappa_n(r)}))g_2(r,X_{r})\, \de r\big|\Bigr\|_{L^p(\Omega)}\\
  &\leq \bar{C}(p, d,\beta,\epsilon)\|f\|_{\Cb}\|g_1\|_{\mathcal{C}^{\alpha,\beta}_b}\|g_2\|_{C^1_b([0,1])} n^{-(1/2+\gamma-\epsilon)}.   
    \end{split}
\end{align}
\end{cor}
\section{Error analysis via a PDE approach}\label{sec:erroranalysis}
\subsection{Some PDE estimate}

\begin{lem}\cite{Pamen2017}\label{lem:pdeestimate}
	For any $\varepsilon \in (0,1)$, there exist $m \in \N$ and $(T_i)_{i=0, \ldots,m}$ such that $0=T_0<\cdots <T_i<T_{i+1}<\cdots<T_m=1$ and for any $i=0,\ldots,m-1$,
	\begin{align*}
	\|\varphi\|_{C_b^{\beta}([0,1])} C_0\cdot (T_{i+1}-T_{i})^{1/2} \leq \varepsilon \text{ and }
	\|f\|_{C_b^{\beta}([0,1])} C_0\cdot (T_{i+1}-T_{i})^{1/2} \leq \frac{1}{4}.
	\end{align*}
 Moreover, for all $\varphi \in C([0,1];C_b^{\beta}(\R^d;\R^d))$, there exists at least one solution $u$ to the backward Kolmogorov equation
	\begin{align*}
	\frac{\partial u}{\partial t}  + \nabla u \cdot f + \frac{1}{2} \Delta u=-\varphi \text{ on } [T_i,T_{i+1}] \times \R^d,~u(T_{i+1},x)=0
	\end{align*}
	of class
	\begin{align*}
	u \in C([T_i,T_{i+1}];C_b^{2,\beta'}(\R^d;\R^d)) \cap C^1([T_i,T_{i+1}];C_b^{\beta'}(\R^d;\R^d))
	\end{align*}
	for all $\beta' \in (0,\beta)$ with
	\begin{align*}
	\|D^2u\|_{C_b^{\beta'}([T_i,T_{i+1}])} \leq M \|\varphi\|_{C_b^{\beta}([T_i,T_{i+1}])}
	\end{align*}
	for some constant $M$ and
	\begin{align*}
	\|\nabla u\|_{C_b^{\beta}([T_i,T_{i+1}])} \leq C_0(T_{i+1}-T_i)^{1/2} \|\varphi\|_{C_b^{\beta}([T_i,T_{i+1}])}
	\end{align*}
	for some constant $C_0$.
\end{lem}
\subsection{The strong convergence of randomised EM}

\begin{thm}\label{thm:main1}
Assume that the drift coefficient $f\in \Cb([0,1]\times\R^d;\R^d)$. Consider the solution $X(t)$  of \eqref{eq:SDE} over $[0,1]$ and its numerical approximation $X_t^\n$ via the randomised numerical scheme \eqref{eqn:randomisedem} at a given stepsize $n^{-1}\in (0,1)$. 
Then for any $p\geq 1$ and $\epsilon\in (0,1/2)$, there exists a positive constant $C$ depending on $m,M, d,p,x_0, \alpha, \beta$ and $\|f\|_{\Cb}$ such that
\begin{align*}
\E\left[\sup_{0 \leq t \leq 1}\left|X(t)-X_t^{(n)}\right|^p \right]  \leq Cn^{  -(1/2+\gamma-\epsilon)p},
\end{align*}
 with $\gamma:=\alpha\wedge (\beta/2).$
\end{thm}

\begin{proof}
For a given $\varepsilon \in (0,1)$, we consider the partition  $(T_i)_{i=0,\ldots,m}$ of closed interval $[0,1]$ which is considered in Lemma \ref{lem:pdeestimate}.
For $l=1,\ldots,d$ and $i=1,\ldots,m$, Lemma \ref{lem:pdeestimate} implies that there exists at least one solution $u_{l,i}$ to the backward Kolmogorov equation:
\begin{align*}
\frac{\partial u_{l,i}}{\partial t}  + \nabla u_{l,i} \cdot f + \frac{1}{2} \Delta u_{l,i}=-f_l \text{ on } [T_{i-1},T_{i}] \times \R^d,~ u_{l,i}(T_{i},x)=0,
\end{align*}
where $f_l$ represents the $l$th coordinate of $f$, and $u_{l,i}$ satisfies,
\begin{align*}
\|\nabla u_{l,i}\|_{C_b^{\beta}[T_{i-1},T_i]}
\leq C_0\cdot(T_i-T_{i-1})^{1/2}\|f\|_{\Cb}
\leq \varepsilon.
\end{align*}
Following \cite{Pamen2017}, for any $t \in [T_{i-1},T_{i}]$, by  It\^o's formula, we have
\begin{align}\label{eqn:XTi}
\int_{T_{i-1}}^t f_l(s,X(s)) \,\de s
=u_{l,i}(T_{i-1},X(T_{i-1}))
-u_{l,i}(t,X(t))
+\int_{T_{i-1}}^t \nabla u_{l,i}(s,X(s)) \,\de B_s,
\end{align}
and
\begin{align}\label{eqn:XnTi}
\int_{T_{i-1}}^t f_l(s,X_s^{(n)}) \,\de s
=&u_{l,i}(T_{i-1},X_{T_{i-1}}^{(n)})
-u_{l,i}(t,X_t^{(n)})
+\int_{T_{i-1}}^t \nabla u_{l,i}(s,X_s^{(n)}) \,\de B_s \nonumber\\
&+\int_{T_{i-1}}^t \nabla u_{l,i}(s,X_s^{(n)}) \cdot \left( f_l(\kappa^\tau_n(s), X_{\kappa_n(s)}^{(n)})-f_l(s, X_s^{(n)}) \right) \,\de s.
\end{align}
We adopt the notation from \cite{Pamen2017} that \begin{align*}
X(t):=\begin{bmatrix}
X_t^1 \\
X_t^2 \\
\vdots \\
X_t^d
\end{bmatrix}
\text{ and }
X_t^{(n)}:=\begin{bmatrix}
X_t^{(n,1)} \\
X_t^{(n,2)} \\
\vdots \\
X_t^{(n,d)}
\end{bmatrix}.
\end{align*}
 It follows from \eqref{eqn:XTi} and \eqref{eqn:XnTi} that for any $l=1,\ldots, d$,
\begin{align}\label{eqn:main1bigdecom1}
\begin{split}
    &X_t^l-X_t^{(n,l)}\\
&=X_{T_{i-1}}^l-X_{T_{i-1}}^{(n,l)}+\int_{T_{i-1}}^t \left( f_l(s,X_s)-
f_l(\kappa^\tau_n(s),X_{\kappa_n(s)}^{(n)})\right)\,\de s\\
&=X_{T_{i-1}}^l-X_{T_{i-1}}^{(n,l)}\\
&\quad+\left( u_{l,i}(T_{i-1},X_{T_{i-1}})-u_{l,i}(T_{i-1},X_{T_{i-1}}^{(n)}) \right)
- \left( u_{l,i}(t,X_{t})-u_{l,i}(t,X_t^{(n)}) \right) \\
&\quad+\int_{T_{i-1}}^t \left( \nabla u_{l,i}(s, X_s) - \nabla u_{l,i}(s,X_s^{(n)}) \right) \,\de B_s\\
&\quad+\int_{T_{i-1}}^t \nabla u_{l,i}(s,X^{\n}_{s}) \cdot  \left(f(s, X_{s}^{(n)})- f(\kappa^\tau_n(s), X_{\kappa_n(s)}^{(n)}) \right) \,\de s\\
&\quad+\int_{T_{i-1}}^t \left( f_l(s, X_s^{(n)}) - f_l(\kappa^\tau_n(s), X_{\kappa_n(s)}^{(n)}) \right)\,\de s.
\end{split}
\end{align}

Since $\|\nabla u_{l,i}\|_{C_b^{\beta}([T_{i-1},T_i])} \leq \varepsilon$, following a similar argument in \cite{Pamen2017}, we have
\begin{align*}
&|X_t^l-X_t^{(n,l)}|\\
&\leq
(1+\varepsilon) \left|X_{T_{i-1}}-X_{T_{i-1}}^{(n)}\right|
+ \varepsilon \left|X_{t}-X_{t}^{(n)}\right|\\
&\quad
+\left| \int_{T_{i-1}}^t \left( \nabla u_{l,i}(s, X_s) - \nabla u_{l,i}(s,X_s^{(n)}) \right) \,\de B_s \right| \\
&\quad+\left|\int_{T_{i-1}}^t \nabla u_{l,i}(s,X^{\n}_{s}) \cdot  \left(f(s, X_{s}^{(n)})- f(\kappa^\tau_n(s), X_{\kappa_n(s)}^{(n)}) \right) \,\de s\right|\\
&\quad+\left|\int_{T_{i-1}}^t \left( f_l(s, X_s^{(n)}) - f_l(\kappa^\tau_n(s), X_{\kappa_n(s)}^{(n)}) \right)\,\de s\right|.
\end{align*}
For $p \geq 2$, using Jensen's and H\"older inequalities, we have
\begin{align}\label{eqn:main1diff}
\begin{split}
    &\left|X_t-X_t^{(n)}\right|^p\\
&\leq
d^{p/2} 5^{p-1} (1+\varepsilon)^p \left|X_{T_{i-1}}-X_{T_{i-1}}^{(n)}\right|^p
+ d^{p/2} 5^{p-1} \varepsilon^p \left|X_{t}-X_{t}^{(n)}\right|^p \\
&\quad+ d^{p/2-1} 5^{p-1} \sum_{l=1}^d \left| \int_{T_{i-1}}^t \left( \nabla u_{l,i}(s, X_s) - \nabla u_{l,i}(s,X_s^{(n)}) \right) \,\de B_s \right|^p \\
&\quad+ d^{3p/2-1} 5^{p-1}  \sum_{l=1}^d\left|\int_{T_{i-1}}^t \nabla u_{l,i}(s,X^{\n}_{s}) \cdot  \left(f(s, X_{s}^{(n)})- f(\kappa^\tau_n(s), X_{\kappa_n(s)}^{(n)}) \right) \,\de s\right|^p \\
&\quad+  d^{3p/2-1} 5^{p-1}  \sum_{l=1}^d\left|\int_{T_{i-1}}^t \left( f_l(s, X_s^{(n)}) - f_l(\kappa^\tau_n(s), X_{\kappa_n(s)}^{(n)}) \right)\,\de s\right|^p.
\end{split}
\end{align}
Since $\varepsilon >0$ is arbitrary, let us fix $\varepsilon$ such that $c(p,d,\varepsilon):=d^{p/2-1} 5^{p-1} \varepsilon^p<1$. 

For the stochastic integral term, for any $t\in [T_{i-1},T_i]$, we have from Burkholder-Davis-Gundy's inequality that
\begin{align}\label{eqn:main1estimate1}
    \begin{split}
&         \sum_{i=1}^d \E\left[\sup_{T_{i-1}\leq v\leq t}\left| \int_{T_{i-1}}^v \left( \nabla u_{l,i}(s, X_s) - \nabla u_{l,i}(s,X_s^{(n)}) \right) \,\de B_s \right|^p\right]\\
 &\leq   C(p,d) T^{\frac{p}{2}-1} \sum_{i=1}^d \int_{T_{i-1}}^t \E\left[  \left| \nabla u_{l,i}(s, X_s) - \nabla u_{l,i}(s,X_s^{(n)}) \right|^p \right]\,\de s,
    \end{split}
\end{align}
where $C(p,d)$ is the constant from Burkholder-Davis-Gundy's inequality.

Now we will make use of Corollary \ref{cor:qbsup} to quantify the last two terms. For the last term, as $f_l\in \Cb([0,1]\times \R^d)$, for any $\epsilon\in (0,1/2)$ there exists a constant $\bar{C}(p, d,\beta,\epsilon)$ that depends on $p$, $\beta$ and $\epsilon$ such as
 \begin{align}\label{eqn:main2term2}
    \begin{split}
             &\E\left[\sup_{T_{i-1}\leq u\leq t}\left|\int_{T_{i-1}}^t \left( f_l(s, X_s^{(n)}) - f_l(\kappa^\tau_n(s), X_{\kappa_n(s)}^{(n)}) \right)\,\de s\right|^p\right]  \\
             &\leq \bar{C}(p, d,\beta,\epsilon) \|f\|_{\Cb}^p n^{-(1/2+\gamma-\epsilon)p}.
    \end{split}
\end{align}
From Lemma \ref{lem:pdeestimate} we know that $\partial u_{l,i} /\partial x^k \in C([0,1];C^1_b(\R^d))$ with $\|\partial u_{l,i} /\partial x^k \|_{C^1_b([0,1])}\leq M\epsilon$. Thus applying Corollary \ref{cor:qbsup} to $f_k\in \Cb([0,1]\times \R^d)$ and $\partial u_{l,i}/\partial x^k\in C([0,1];C^1_b(\R^d))$ yields

\begin{align}\label{eqn:main2term1}
    \begin{split}
             &\E\left[\sup_{T_{i-1}\leq u\leq t}\left|\int_{T_{i-1}}^u \nabla u_{l,i}(s,X^{\n}_{s}) \cdot  \left(f(s, X_{s}^{(n)})- f(\kappa^\tau_n(s), X_{\kappa_n(s)}^{(n)}) \right) \,\de s\right|^p\right]  \\
            &\leq \sum_{k=1}^d d^{p-1} \E\left[\sup_{T_{i-1}\leq u\leq t}\left|\int_{T_{i-1}}^u  \frac{\partial u_{l,i}(s,X^{\n}_{s})}{\partial x^k}  \left(f_k(s, X_{s}^{(n)})- f_k(\kappa^\tau_n(s), X_{\kappa_n(s)}^{(n)}) \right) \,\de s\right|^p\right]  \\
             &\leq \bar{C}(p, d,\beta,\epsilon) d^p \varepsilon^p M^p \|f\|_{\Cb}^p n^{-(1/2+\gamma-\epsilon)p},
    \end{split}
\end{align}
for any $\epsilon\in (0,1/2)$.

 Taking the supremum and then expectation on both sides of \eqref{eqn:main1diff}, the contrain on $\varepsilon$, we have from estimate \eqref{eqn:main1estimate1} to \eqref{eqn:main2term1} that
\begin{align*}
&\E\left[\sup_{T_{i-1} \leq u \leq t}\left|X_u-X_u^{(n)} \right|^p \right]\\
&\leq
\frac{d^{p/2} 5^{p-1} (1+\varepsilon)^p}{(1-c(p,d,\varepsilon))} \E\left[\left|X_{T_{i-1}}-X_{T_{i-1}}^{(n)}\right|^p \right]  \\
&+ \frac{d^{p/2-1} 5^{p-1} C(p,d)  }{(1-c(p,d,\varepsilon))} \sum_{i=1}^d \int_{T_{i-1}}^t \E\left[  \left| \nabla u_{l,i}(s, X_s) - \nabla u_{l,i}(s,X_s^{(n)}) \right|^p \right]\,\de s \\
&+ \frac{d^{3p/2-1} 5^{p-1}\bar{C}(p, d,\beta,\epsilon) (1+d^p \varepsilon^p M^p)\  \|f\|_{\Cb}^p d^p \varepsilon^p M^p }{(1-c(p,d,\varepsilon))} \\
&\leq C_1 \E\left[\left|X_{T_{i-1}}-X_{T_{i-1}}^{(n)}\right|^p \right]
+C_2 \int_{T_{i-1}}^t \E\left[ \sup_{T_{i-1} \leq u \leq s }\left|X_u-X_u^{(n)} \right|^p \right]\,\de s
+\frac{C_3}{n^{p (1/2+\gamma-\epsilon)}}.
\end{align*}
The remaining part follows exactly the same as in the proof of Theorem 2.11 of \cite{Pamen2017}. 
\end{proof}

\begin{rmk}\label{rmk:optimal}
    Note that the optimal order of convergence for a strong approximation of SDE \eqref{eq:SDE} is $1/2+\gamma$. This can be easily derived from two facts: first, the optimal order of convergence for a strong approximation of additive SDE with $f\in C^\beta_b(\R^d;\R^d)$ is shown to be $1/2+\beta/2$; secondly,  in the ODE case, the maximum order of convergence of randomized algorithms is
known to be equal to $1/2+\alpha$
under the assumption that $f\in \mathcal{C}^{\alpha,1}$, see \cite{Heinrich2008}.
\end{rmk}

 \section*{Acknowledgments}
YW would like to acknowledge the support of the Royal Society through the International Exchanges scheme IES\textbackslash R3\textbackslash 233115.
\appendix

\section{Useful estimate from heat kernel \cite{Butkovsky2021}}
Let $p_t$, $t>0$, be the density of a $d$-dimensional vector with independent Gaussian components each of mean zero and variance $t$:
\begin{equation}\label{eq:p-def}
p_t(x)=\frac{1}{(2\pi t)^{d/2}}\exp\Bigl(-\frac{|x|^2}{2t}\Bigr),\quad x\in\R^d.
\end{equation}
For a measurable function $g\colon\R\times\R^d\to\R$ we write $\cP_t g(r,\cdot):=p_t\ast g(r,\cdot)$ for $r\in \R$, and occasionally we denote by $p_0$ the Dirac delta function. The first statement provides a number of technical bounds related to the Brownian motion.
\begin{prop}\cite{Butkovsky2021}\label{prop:fractional}
Let $p\ge1$. The $d$-dimensional process $B$ has the following properties:
\begin{enumerate}
\item $\|B_t-B_s\|_{L^p(\Omega; \R^d)}= \bar{C}(p,d) |t-s|^{1/2}$, for all $0\leq s\leq t\leq 1$;
\item $\E^s [g(r,B_t)]=\cP_{|s-t|}g(r,\E^s[B_t])$, for all $0\leq s\leq t\leq 1$ and $r\in \R$;
\item $\|\E^s[B_t]-\E^s[B_u]\|_{L^p(\Omega; \R^d)}\leq \bar{C}(p,d)|t-u||t-s|^{-1/2}$, for all $0\leq s\leq u\leq t$ such that $|t-u|\le |u-s|$;
\end{enumerate}
\end{prop}

The next statement gives the  heat kernel bounds which are necessary for the proofs of the quadratic bounds in Section \ref{sec:lemmas}. 
\begin{prop}\label{prop:HK}
Let $g\in \Cb$, $\alpha, \beta\le 1$ and $\eta \in[0,1]$. The following  hold:
\begin{enumerate}
\item
There exists $\bar{C}(d, \beta, \eta)$ such that 
\begin{equation}\label{eqn:A2single} \|\cP_tg(r,\cdot)\|_{C^\eta(\R^d; \R^d)}\le \bar{C}(d, \beta, \eta)t^{\frac{(\beta-\eta)\wedge0}{2}} \|g(r,\cdot)\|_{C^\beta(\R^d;\R^d)},
\end{equation}
for all $t,r\in(0,1]$.
\item 
For all $\delta \in (0,1]$ with $\delta\ge\frac\beta2-\frac\eta2$, there exists $\bar{C}(d, \beta, \eta, \delta)$ such that 
\begin{equation}\label{eqn:A2difference}
    \|\cP_tg(r,\cdot)-\cP_sg(r,\cdot)\|_{C^\eta(\R^d;\R^d)}\leq \bar{C}(d, \beta, \eta, \delta) \|g(r,\cdot)\|_{C^{\beta}(\R^d;\R^d)} s^{\frac\beta2-\frac\eta2-\delta}(t-s)^{\delta},
\end{equation}
for all $0\le s\leq t \le 1$ and $r\in [0,1]$.
\end{enumerate}
\end{prop}

\section{An alternative proof of Lemma \ref{lem:qb3_discrete}}\label{sec:proof}
\begin{proof}[The proof of Lemma \ref{lem:qb3_discrete}]
 When $t-s < n^{-1}$, it is easy to get
\begin{align*}
  &\Bigl\|\int_s^t \big(g_1(r,B_{\kappa_n(r)})-g_1(\kappa_n^\tau(r),B_{\kappa_n(r)})\big)g_2(\kappa_n(r),B_{\kappa_n(r)})\, \de r\Bigr\|_{L^p(\Omega)}\\
  &\leq   \|g_1\|_{\mathcal{C}^{\alpha,\beta}_b}\|g_2\|_{\infty} |t-s|n^{-\alpha}\leq \|g_1\|_{\mathcal{C}^{\alpha,\beta}_b}\|g_2\|_{\infty} |t-s|^{1/2+\epsilon}n^{-(1/2+\alpha-\epsilon)}. 
\end{align*}

When $n^{-1}\leq t-s$, for simplicity, take $s=k_1/n$ and $t=k_2/n$, with $k_1\leq k_2$ and $k_1, k_2\in \N$ . It suffices to show that 
\begin{align}\label{eqn:errRiediscrete}
         \begin{split}
              &  \Big\| \sup_{k\in \{k_1,\ldots,k_2\}} \Big|
      \int_{t_{k_1}}^{t_k}   \left( g_1(r, B_{\kappa_n(r)})-g_1(\kappa^\tau_n(r), B_{\kappa_n(r)}) \right)g_2(\kappa_n(r),B_{\kappa_n(r)}) \diff{r}
      \Big| \, \Big\|_{L^p( \Omega)}\\
      &\leq \bar{C}(p)\|g_1\|_{\mathcal{C}^{\alpha,\beta}_b}\|g_2\|_{\infty}n^{-(\frac{1}{2} + \alpha)} \sqrt{t-s}
      .
         \end{split}
     \end{align}
Eqn. \eqref{eqn:errRie} can be derived immediately given estimate \eqref{eqn:errRiediscrete}.
  
  Let us therefore fix an arbitrary realization $\omega \in \OB$.
  Then for every $k \in \{k_1+1,\ldots,k_2\}$ we obtain
  \begin{align}\label{eqn:conditionalexp}
  \begin{split}
         & \E_\tau \left [\int_{t_{k-1}}^{t_k}   g_1(\kappa^\tau_n(r), B_{\kappa_n(r)})\,g_2(\kappa_n(r),B_{\kappa_n(r)}) \diff{r} \Big|\mathcal{F}^\tau_{k-1} \right]  \\
    &=\E_\tau \left [n^{-1} g_1\!\left(t_{k-1}+ \tau_k n^{-1}, B_{t_{k-1}}(\omega)\right)\, g_2\!\left(t_{k-1},B_{t_{k-1}}(\omega)\right)  \Big|\mathcal{F}^\tau_{k-1} \right] \\
      & =\int_0^1 n^{-1}g_1\!\left(t_{k-1}+ v n^{-1}, B_{t_{k-1}}(\omega)\right)\, g_2\!\left(t_{k-1},B_{t_{k-1}}(\omega)\right) \,\de v  \\
&    = \int_{t_{k-1}}^{t_k} g_1\!\left(r, B_{\kappa_n(r)}(\omega)\right)\,g_2\!\left(\kappa_n(r),B_{\kappa_n(r)}(\omega)\right)  \diff{r},
  \end{split}
  \end{align}
  due to $\tau_k \sim \mathcal{U}(0,1)$ and independent of $\mathcal{F}^\tau_{k-1} $. 
  
  Next, define filtration $\mathcal{G}^n_{\bar{n}}:=\F_{\bar{n}+k_1}^n$ for $\bar{n} \in \{1,\ldots,k_2-k_1\}$, and define a discrete-time
  error process $(E^{\bar{n}})_{\bar{n} \in \{0,1,\ldots,k_2-k_1\}}$ by setting $E^0 \equiv 0$ and for $\bar{n} \in \{1,\ldots,k_2-k_1\}$ setting
  \begin{align*}
          E^{\bar{n}} := 
 \int_{t_{k_1}}^{t_{k_1+\bar{n}}}   \left( g_1(r, B_{\kappa_n(r)})-g_1(\kappa^\tau_n(r), B_{\kappa_n(r)}) \right) g_2(\kappa_n(r),B_{\kappa_n(r)}) \diff{r},
  \end{align*}
  which is evidently an real-valued random variable on the product
  probability space $(\Omega, \F, \P)$. In particular, $(E^{\bar{n}})_{\bar{n} \in
  \{0,1,\ldots,k_2-k_1\}} \subset L^p(\Omega)$. Moreover,
  for each fixed $\omega \in \OB$ we have
  that $E^{\bar{n}}(\omega,\cdot) \colon \Omega_\tau \to \R$ is
  $\mathcal{G}_{\bar{n}}^\tau$-measurable. Further, for
   $\bar{n}\in \{0,1,\ldots,k_2-k_1-1\}$, it holds true that
  \begin{align*}
    &\E_\tau[ E^{\bar{n}+1}(\cdot,\omega) | \mathcal{G}^{\tau}_{\bar{n}}]=
    E^{\bar{n}}(\cdot,\omega) 
  \end{align*}
  because of Eqn. \eqref{eqn:conditionalexp}.
  
  Consequently, for every $\omega \in \OB$ 
  the error process $(E^{\bar{n}}(\cdot,\omega))_{\bar{n} \in
  \{0,1,\ldots,k_2-k_1\}}$ is an $(\mathcal{G}^\tau_{\bar{n}})_{\bar{n} \in
  \{0,1,\ldots,k_2-k_1\}}$-adapted $L^p(\Omega_\tau;\R^d)$-martingale.
  Thus,  the discrete-time version of the Burkholder-Davis-Gundy inequality 
  (see Theorem \ref{th:discreteBDG}) is applicable and yields  
  \begin{align*}
    \big\| \max_{\bar{n} \in
  \{0,1,\ldots,k_2-k_1\}} | E^{\bar{n}}(\cdot,\omega)|
    \big\|_{L^p(\Omega_\tau)}
    \le C_p \big\| [E(\cdot,\omega)]^{\frac{1}{2}}_{k_2-k_1}
    \big\|_{L^p(\Omega_\tau)}
    \quad \text{ for every } \omega \in \OB.
  \end{align*}
  After inserting the quadratic variation $[E(\omega,\cdot)]_{k_2-k_1}$, taking the
  $p$-th power and integrating with respect to $\P_W$ we arrive at
  \begin{align*}
      &\big\| \max_{\bar{n} \in
  \{0,1,\ldots,k_2-k_1\}} | E^{\bar{n}}|\big\|_{L^p(\Omega)}^p\\
  & = \int_{\Omega_W} \big\| \max_{\bar{n} \in
  \{0,1,\ldots,k_2-k_1\}} | E^{\bar{n}}(\cdot,\omega)| \big\|_{L^p(\Omega_\tau)}^p \diff{\P}_W(\omega) \\
            & \le C_p^p \int_{\Omega_W} \Big\| \Big(\sum_{k = k_1+1}^{k_2}
      \Big| \int_{t_{k-1}}^{t_k}  \Big( g_1(r, B_{\kappa_n(r)})-g_1(\kappa^\tau_n(r), B_{\kappa_n(r)}) \Big)\\
      & \qquad\qquad \times g_2(\kappa_n(r),B_{\kappa_n(r)})\diff{r}  \Big|^2 \Big)^{\frac{1}{2}} 
      \Big\|_{L^p(\Omega_\tau)}^p \diff{\P}_W(\omega)\\     
      & \leq C_p^p\|g_1\|_{\mathcal{C}^{\alpha,\beta}_b}^p\|g_2\|_{\infty}^p \Big(\sum_{k = k_1+1}^{k_2} n^{-2} n^{-2\alpha}  \Big)^{p/2}\\
      &\leq C_p^p\|g_1\|_{\mathcal{C}^{\alpha,\beta}_b}^p\|g_2\|_{\infty}^p (t-s)^{p/2} n^{-(1/2+\alpha)p}. 
  \end{align*}
  For an arbitrary interval $[s,t]$ satisfying $n^{-1}\leq |t-s|$, Eqn. \eqref{eqn:errRie} follows by interpolation.
\end{proof}

\end{document}